\documentclass[11pt, reqno]{amsart}
\usepackage{graphicx, amssymb, amsmath, amsthm}
\usepackage[utf8]{inputenc}
\usepackage{epsfig}
\usepackage{hyperref}
\usepackage{comment}
\usepackage{mathrsfs}
\numberwithin{equation}{section}
\usepackage{amsthm}
\usepackage{subfigure}
\usepackage{tikz}
\usepackage{here}
\usepackage{float}
\usepackage{pifont}
\usepackage{enumitem}
\usetikzlibrary{matrix,arrows}
\usetikzlibrary{shapes}
\usetikzlibrary{calc}
\usetikzlibrary{arrows}
\usetikzlibrary{decorations.pathreplacing,decorations.markings}
\usepackage[all]{xy}

\usepackage{tikz-cd}

\usetikzlibrary{patterns}

\newtheorem{theorem}{Theorem}[section]

\newtheorem{lemma}[theorem]{Lemma}

\newtheorem{corollary}[theorem]{Corollary}
\newtheorem{conjecture}[theorem]{Conjecture}

\theoremstyle{definition}
\newtheorem{definition}[theorem]{Definition}
\newtheorem{remark}[theorem]{Remark}

\makeatletter
\newcommand{\Extend}[5]{\ext@arrow0099{\arrowfill@#1#2#3}{#4}{#5}}
\makeatother

\DeclareMathOperator{\Lip}{Lip}
\DeclareMathOperator{\dist}{dist}
\DeclareMathOperator{\diam}{diam}

\begin{document}

\title[Rational homology vanishing theorem]{A note on rational homology vanishing theorem for hypersurfaces in aspherical manifolds}
\author{Shihang He}

\address[Shihang He]{Key Laboratory of Pure and Applied Mathematics,
School of Mathematical Sciences, Peking University, Beijing, 100871, People's Republic of China}
\email{hsh0119@pku.edu.cn}
\author{Jintian Zhu}
\address[Jintian Zhu]{Institute for Theoretical Sciences, Westlake University, 600 Dunyu Road, 310030, Hangzhou, Zhejiang, People's Republic of China}
\email{zhujintian@westlake.edu.cn}

\maketitle
\begin{abstract}
In this note, Gromov's reduction \cite{Gro20}, from the aspherical conjecture to the generalized filling radius conjecture, is generalized to the smooth $\mathbb Q$-homology vanishing conjecture (see Remark \ref{Rmk: smooth vanishing}) in the case of hypersurface. In particular, we can show that any continuous map from a closed $4$-manifold admitting positive scalar curvature to an aspherical $5$-manifold induces zero map between $H_4(\cdot,\mathbb Q)$. As a corollary, we obtain the following aspherical splitting theorem: if a complete orientable aspherical Riemannian $5$-manifold has nonnegative scalar curvature and two ends, then it splits into the Riemannian product of a closed flat manifold and the real line.
\end{abstract}\section{Introduction}

A basic problem in the research of scalar curvature is to classify all closed manifolds with positive scalar curvature. The problem is completely solved in dimension two and three: a closed orientable surface admitting positive scalar curvature is the $2$-sphere and closed orientable $3$-manifolds admitting positive scalar curvature are those having prime decomposition in the form of 
$$(\mathbb S^3/\Gamma_1)\#\cdots \#(\mathbb S^3/\Gamma_k)\#l(\mathbb S^2\times \mathbb S^1).$$
In dimensions four and five only partial classification result was achieved by Chodosh-Li-Liokumovich \cite{CLL23} assuming the manifold to be sufficiently connected.

Although no complete classification has been obtained in dimension no less than four, it is believed that 
a closed orientable manifold admitting positive scalar curvature should have a finite cover constructed from spheres. The explicit statement was raised by Gromov \cite[page 96]{Gro23} as the following {\it $\mathbb Q$-homology vanishing conjecture}:
\begin{conjecture}\label{Conj: Q vanishing}
    Let $M^n$ be a closed manifold admitting positive scalar curvature. For any continuous map $f:M\to X$ mapping $M$ into an aspherical topological space $X$, we have $f_*([M])=0$ in $H_n(X,\mathbb Q)$.
\end{conjecture}
\begin{remark}\label{Rmk: smooth vanishing}
    For convenience, if $X$ turns out to be a smooth aspherical manifold rather than just a topological space, we use the name {\it smooth $\mathbb Q$-homology vanishing conjecture} to call this special version of Conjecture \ref{Conj: Q vanishing}.
\end{remark}

It is worth mentioning that Conjecture \ref{Conj: Q vanishing} is closely related to the strong Novikov Conjecture when $M$ is spin. For pioneering works in this direction one may refer to \cite{Ros83}, and for a detailed discussion of related topics one could see \cite{Han12} and references therein.

In the case when $n=2$ or $3$, the $\mathbb Q$-homology vanishing conjecture is well-known to be true with the help of the topological classification for closed manifolds with positive scalar curvature (see Appendix \ref{Append} for details). When $X$ is a Cartan-Hadamard manifold, the conjecture was shown to be true in \cite{GL83} by Gromov-Lawson under the assumption that $M$ is spin. In \cite{SY87}, Schoen-Yau mentioned they could obtain the same result without the spin condition. 

As a special case of the smooth $\mathbb Q$-homology vanishing conjecture, the aspherical conjecture \cite{Gro86, SY87} asserts that any closed aspherical manifold cannot admit positive scalar curvature, which was recently verified independently by Gromov \cite{Gro20} and by Chodosh-Li \cite{CL20} up to dimension five (also see previous contribution from \cite{GL83, SY87} in lower dimensions). In particular, Gromov \cite{Gro20} made a reduction from the aspherical conjecture in dimension $n$ to the validity of the generalized filling radius conjecture in dimension $(n-2)$. Recall that for any orientable closed Riemannian manifold $(M,g)$ its {\it filling radius} is defined to be
$$r_f(M,g):=\inf\{r>0|\mathcal K_*:H_n(M,\mathbb Z)\to H_n(U_r(\mathcal K(M)),\mathbb Z)\mbox{ is zero map}\},$$
where $\mathcal K$ is the Kuratowski embedding given by
$$\mathcal K:M\to L^\infty(M),\,p\mapsto \dist_g(p,\cdot)$$
and $U_r(\mathcal K(M))$ is denoted to be the $r$-neighborhood of $\mathcal K(M)$ in $L^\infty(M)$.
As a preparation, we also recall the $\mathbb T^l$-stabilized scalar curvature from \cite{Gro20}.
\begin{definition}\label{Defn: stablized curvature}
Let $(M,g)$ be a closed Riemannian manifold and $R$ be a smooth function on $M$. We say that $M$ has $\mathbb T^l$-stabilized scalar curvature $R$ if there are $l$ positive smooth functions $v_1,v_2,\ldots,v_l$ on $M$ such that $(M\times \mathbb T^l,g+\sum_i v_i^2\mathrm d\theta_i^2)$ has scalar curvature $\tilde R$ in $M\times \mathbb T^l$, where $\tilde R$ is $\mathbb T^l$-invariant extension of $R$.
\end{definition} 
Then the {\it generalized filling radius conjecture} is stated as following
\begin{conjecture}\label{Conj: fill}
    Let $(M^n,g)$ be an orientable closed Riemannian manifold with $\mathbb T^l$-stabilized scalar curvature $R\geq 1$. Then the filling radius of $(M,g)$ is no greater then $C(n)$.
\end{conjecture}
The generalized filling radius conjecture is currently known to be true up to dimension three \cite{Gro83, Gro20, CL20,LM23}.
In this note, we plan to generalize Gromov's reduction to the above smooth $\mathbb Q$-homology vanishing conjecture in the case of hypersurface (i.e. $\dim X=\dim M+1$).
Our main theorem is stated as following

\begin{theorem}\label{Thm: vanishing}
    Let $M^n$, $3\leq n\leq 7$, be a closed manifold admitting a smooth metric with positive scalar curvature and $X^{n+1}$ be an aspherical manifold. If the generalized filling radius conjecture is true in dimension $(n-1)$, then for any continuous map $f:M\to X$ we have $f_*([M])=0$ in $H_n(X,\mathbb Q)$.
\end{theorem}

In particular, we are able to answer Conjecture \ref{Conj: Q vanishing} affirmatively for aspherical manifold $X$ of dimension 5.

\begin{corollary}\label{Cor: 5D}
    Let $X$ be an aspherical $5$-manifold. Then there is no smooth metric with positive scalar curvature on any closed $k$-manifold $M$ representing a non-zero homology class $\beta\in H_k(X,\mathbb Q)$ for all $2\leq k\leq 5$. 
\end{corollary}

As an immediate application of Corollary \ref{Cor: 5D}, we are able to prove the following {\it aspherical splitting theorem}:
\begin{theorem}\label{Thm: splitting}
    Assume that $(X^{n+1},g)$, $n\leq 4$, is a complete orientable aspherical Riemannian manifold with two ends and with nonnegative scalar curvature. Then it splits into the Riemannian product $(X_0,g_0)\times \mathbb R$, where $(X_0,g_0)$ is a closed flat Riemannian manifold.
\end{theorem}

We point out that a similar phenomenon also appears in the research of manifolds with scalar curvature bounded from below by a negative constant (see \cite[Theorem 1.7]{HHLS23} for instance). Our Corollary \ref{Cor: 5D} could be used to improve \cite[Theorem 1.7]{HHLS23} to hold in dimension five.

\section*{Acknowledgement}
Both authors are supported by National Key R\&D Program of China Grant 2020YFA0712800 and the second-named author is also supported by the start-up fund from Westlake University. We are very grateful to Prof. Yuguang Shi for his constant encouragement and enlightening conversations on the topic. The authors are grateful to anonymous referees for their careful reading and helpful suggestions.

\section{The proof}

We begin with the following basic topological lemma.
\begin{lemma}\label{lem: topological construction}
    Given a continuous map $i:A\to Y$, where $Y$ is a contractible manifold and $A$ is a finite CW complex, for any $k\in \mathbb{Z}_+$ we can extend the map $i$ to a new continuous map $F:Z\to Y$, where $Z$ is a $k$-connected finite CW complex containing $A$ as a subcomplex and $F|_A = i$.
\end{lemma}
\begin{proof}
        Since $A$ is a finite CW complex, it must be compact, which implies that the fundamental group  $\pi_1(A)$ is finitely generated. Then it is standard to construct a simply-connected finite CW complex $Z_1$ by attaching finitely many 2-cells to $A$ to eliminate $\pi_1(A)$. It is well-known that the cell-attaching procedure can further produce $k$-connected finite CW complex by induction. Assume that we have already obtained an $i$-connected finite CW complex $Z_i$. From the Hurewicz Theorem we know 
        $$\pi_{i+1}(Z_i) = H_{i+1}(Z_i)$$
        and in particular $\pi_{i+1}(Z_i)$ is finitely generated. By attaching finitely many $(i+1)$-cells to $Z_i$ to eliminate $\pi_{i+1}(Z_i)$ we can construct an $(i+1)$-connected finite CW complex $Z_{i+1}$. 
        
        Take $Z = Z_k$. Then $Z$ is a $k$-connected finite CW complex $Z$ containing $A$ as a subcomplex. 
        Since $Y$ is contractible, $i:A\longrightarrow Y$ can be extended to $F:Z\longrightarrow Y$ by the extension lemma \cite[Lemma 4.7]{Hat02}.
\end{proof}

Next we prove a quantitative filling lemma for homotopy groups in the universal covering of aspherical manifolds.
\begin{lemma}\label{lem1}
    Let $(X^m,g)$ be an aspherical manifold and $\pi:(\bar{X}^{m},\bar g)\longrightarrow (X^{m},g)$ be the universal Riemannian covering. Let $K\subset X$ be a compact set and $\bar{K} = \pi^{-1}(K)$. Then for any $r>0$ and $k\in \mathbb{Z}_+$, there is a constant $R=R(r,K,k)>0$ with the following property: for any continuous map $\alpha: \mathbb S^k\longrightarrow \bar{X}$, such that the image of $\alpha$ is contained in $\bar{K}$ with $diam(\alpha)\le r$, there is a continuous map $\beta: \mathbb B^{k+1}\longrightarrow\bar{X}$ with $diam(\beta)\le R$ and $\beta|_{\partial \mathbb B^{k+1}} = \alpha$, where we use $diam(\cdot)$ to denote the diameter of the image of a map.
\end{lemma}
\begin{proof}
    Fix a point $q\in\bar{K}$ and a point $p$ in the image of the map $\alpha$. Since $\bar{X}$ is a regular covering space, we can find a deck transformation $\Phi$ such that $\dist(\Phi(p),q)\le D_0:= \diam(K)$ and so the image of $\Phi\circ\alpha$ is contained in the ball $ B_q(r+D_0)$. Take $A$ to be the closure of a smooth domain containing $B_q(r+D_0)$. It is well-known that $A$ admits a triangulation and so $A$ becomes a finite CW complex. From Lemma \ref{lem: topological construction} we can extend the inclusion map $i:A\to\bar X$ to a new continuous map $F:Z\to \bar X$, where $Z$ is a $k$-connected finite CW complex containing $A$ as a subcomplex and $F|_A=i$. Since the inclusion $i:A\to i(A)\subset\bar X$ is a homeomorphism, we can find a continuous map $\alpha_0:\mathbb S^k\to A\subset Z$ satisfying $i\circ \alpha_0=\Phi\circ \alpha$. From the $k$-connectedness of $Z$ the map $\alpha_0$ can be extended to a continuous map $\beta_0:\mathbb B^{k+1}\to Z$. Define 
    $$\beta=\Phi^{-1}\circ F\circ \beta_0.$$
    It is easy to verify $\beta|_{\partial \mathbb B^{k+1}}=\alpha$. 
    
    Recall that $Z$ is a finite CW complex. In paticular, $Z$ is compact and its image $F(Z)$ has bounded diameter. Take $R=\diam(F(Z))+1$. Then we have $\diam(F\circ\beta_0)<R$ since the image of $F\circ\beta_0$ is contained in $F(Z)$. Using the fact that deck transformations are isometries, we see $\diam(\beta)\leq R$ and this completes the proof.
\end{proof}

Next we derive the relative filling radius of a Riemannian manifold from its absolute filling radius.
\begin{lemma}\label{lem2}
 Let $(X^{m},g)$ be an aspherical manifold and $\pi:(\tilde{X}^{m},\tilde g)\longrightarrow (X^{m},g)$ be a Riemannian covering. Let $K\subset X$ be a compact set and $\tilde{K} = \pi^{-1}(K)$. Suppose that Conjecture \ref{Conj: fill} is true in dimension $n$.
  Let $(N^n,h)$ be a closed manifold with stabilized scalar curvature $R\ge 1$ and $f: N\longrightarrow \Tilde{K}\subset \Tilde{X}$ a distance decreasing map. If $(\tilde X,\tilde g)$ satisfies  $sys_1(\tilde X,\tilde g)>6C(n)$ in case $\pi_1(\Tilde{X})\ne 0$, where $C(n)$ is the constant from Conjecture \ref{Conj: fill} and
    $$sys_1(\tilde X,\tilde g):=\inf\{Length_{\tilde g}(\gamma):\gamma:\mathbb S^1\to \tilde X,\,[\gamma]\neq 0\in \pi_1(\tilde X)\},$$
    then there is a $(n+1)$-chain $\sigma$ in $\Tilde{X}$ with the following property: $\partial \sigma = f_\#N$ and $\sigma$ is supported in the neighborhood $\lbrace x\in \Tilde{X}, d(x, f(N))\le \tilde C\rbrace$ for some absolute constant $\tilde C=\tilde C(n, X,K)$. 
\end{lemma}
\begin{proof}
    The proof comes from a combination of Lemma \ref{lem1} and the argument of \cite[4.F]{Gro20}. We provide the details for completeness. From Conjecture \ref{Conj: fill} we see $r_f(N,h)\leq C(n)$ and by definition we can find a polyhedral space $\mathcal P$ with $\partial\mathcal P=N$ associated with a continuous map $i:\mathcal P\to L^\infty(N)$ such that
    \begin{itemize}
\item[(p1)] $i|_{\partial\mathcal P}$ is the Kuratowski isometric embedding of $N$;
\item[(p2)] $i(\mathcal P)$ lies in the $C(n)$-neighborhood of $i(\partial\mathcal P)$ in $L^\infty(N)$.
    \end{itemize}
    
    Denote
    $$\epsilon=\frac{1}{12}\left(sys_1(\tilde X,\tilde g)-6C(n)\right)\mbox{ when }\pi_1(\tilde X)\neq 0$$
   and $\epsilon=1$ when $\pi_1(\tilde X)=0$. By subdivision we may assume $diam(i(\Delta))\leq \epsilon$ for every simplex $\Delta$ in $\mathcal P$ and that each $1$-simplex contained in $\partial \mathcal P$ has its length no greater than $\epsilon$.
   
  In the following, we view $f$ as a map $f:\partial\mathcal P\to \tilde X$. Our goal is to construct a continuous map $F:\mathcal P\to\tilde X$ such that $F|_{\partial{\mathcal P}}=f$ and that $F({\mathcal P})$ lies in $\tilde C$-neighborhood of $f(N)$ for some absolute constant $\tilde C$. 
  
For convenience, we denote $\mathcal P_l$ to be the $l$-skeleton of $\mathcal P$. For any vertex $v\in \mathcal P_0$ we take the nearest point $p_v$ of $i(v)$ on $i(\partial\mathcal P)$ in $L^\infty(N)$ and define $$F:{\mathcal P}_0\to \tilde X,\,v\mapsto f\circ (i|_{\partial\mathcal P})^{-1}(p_v).$$ It is clear that we have the distance estimate $\dist_{L^\infty}(i(v),p_v)\leq C(n)$ from the property (p2) above. 
   
   Next we want to extend $F$ to be a map $F:\mathcal P_1\to\tilde X$. Denote $\overline{v_1v_2}$ to be an arbitrary $1$-simplex in $\mathcal P_1$. There are two possibilities: 
   \begin{itemize}
       \item if $\overline{v_1v_2}$ lies in $\partial\mathcal P$, then we define $F|_{\overline{v_1v_2}}=f|_{\overline{v_1v_2}}$;
       \item otherwise  we take one minimizing geodesic $\gamma_{v_1v_2}:[0,1]\to \partial\mathcal P$ which connects $(i|_{\partial\mathcal P})^{-1}p_{v_1}$ and $(i|_{\partial\mathcal P})^{-1}p_{v_2}$ and then define $$F|_{\overline{v_1v_2}}=(f\circ\gamma_{v_1v_2})([0,1]).$$
   \end{itemize} 
   
 We claim that the length of all paths $f|_{\overline{v_1v_2}}$ cannot exceed $2C(n)+\epsilon$. In the first case, $\overline{v_1v_2}$ is a $1$-simplex contained in $\partial\mathcal P$ and so its length is no greater than $\epsilon$, where the claim follows from the distance-decreasing property of $f$. Now we focus on the second case.  From the property (p1)  we know
   \[
\dist_N((i|_{\partial\mathcal P})^{-1}p_{v_1},(i|_{\partial\mathcal P})^{-1}p_{v_2})
   =\dist_{L^\infty}(p_{v_1},p_{v_2}).
   \]
   Using the triangle inequality we obtain
   \[
   \begin{split}
   \dist_{L^\infty}(p_{v_1},p_{v_2})&
   \leq \dist_{L^\infty}(p_{v_1},i(v_1))\\
   &\qquad\qquad+\dist_{L^\infty}(p_{v_2},i(v_2))+\dist_{L^\infty}(i(v_1),i(v_2))\\
   &\leq 2C(n)+\epsilon.
   \end{split}\]
  This yields that $\gamma_{v_1v_2}$ has its length no greater than $2C(n)+\epsilon$. Since $f$ is distance-decreasing, the length of the path $(f\circ\gamma_{v_1v_2})([0,1])$ is no greater than $2C(n)+\epsilon$ as well and this completes the proof for the claim. 
   
  We show that the map $F$ can be further extended on $\mathcal P_2$ using the previous length estimate. Take any $2$-simplex $\Delta_2=\overline{v_1v_2v_3}$ in $\mathcal P_2$. If $\Delta_2$ is contained in $\partial\mathcal P$, then we can define $F|_{\Delta_2}=f|_{ \Delta_2}$ directly. Otherwise, we have to construct the desired extension. From previous length estimate we see that the loop $F:\partial\Delta_2\to \tilde X$ has length no greater than $6C(n)+3\epsilon$. Either when $\pi_1(\tilde X)=0$ or when $sys_1(\tilde X,\tilde g)>6C(n)+3\epsilon$ in case $\pi_1(\tilde X)\neq 0$, the loop $F:\partial\Delta_2\to \tilde X$ bounds a disk and so it can be lifted to a loop $\bar\gamma:\partial\Delta_2\to \bar X$ in the universal cover $(\bar X,\bar g)$ of $(\tilde X,\tilde g)$. From Lemma \ref{lem1} we can find a disk $\bar F:\Delta_2\to \bar X$ such that $\partial\bar F=\bar\gamma$ and that $\bar F(\Delta_2)$ lies in some $\tilde C_1$-neighborhood of $\bar\gamma(\partial\Delta_2)$ in $(\bar X,\bar g)$ for some universal constant $\tilde C_1=\tilde C_1(n,X,K)$. Then we can define
  $F|_{\Delta_2}=\pi\circ \bar F$ and this gives the definition for $F:\mathcal P_2\to \tilde X$.
Since $F(\mathcal P_1)$ is contained in $f(N)$, we conclude that $F(\mathcal P_2)$ lies in $\tilde C_1$-neighborhood of $f(N)$. 

   Now we can take a new compact subset 
   $$K'=\{x\in X:\dist(x,K)\leq \tilde C_1\}$$ to guarantee $F({\mathcal P}_2)\subset \tilde K'$, where $\tilde K'=\pi^{-1}(K')$. Since the boundary $\partial\Delta_k$ of a $k$-cell is simply-connected when $k\geq 3$. The lifting-and-filling argument works when we extend $F:\mathcal P_{k-1}\to\tilde X$ to a map $F:\mathcal P_k\to\tilde X$ when $k\geq 3$. Finally we end up with a continuous map $F:\mathcal P\to \tilde X$ with $F|_{\partial\mathcal P}=f$. Since the construction is only repeated for finitely many times depending on the dimension $n$, we conclude that $F(\mathcal P)$ lies in some $\tilde C$-neighborhood of $f(N)$ in $(\tilde X,\tilde g)$ for some universal constant $\tilde C=\tilde C(n,X,K)$.
\end{proof}
\begin{remark}\label{Rmk: compact}
    From above proof we only need $sys_1(\tilde K,\tilde g)>6C(n)$ since we have $F(\partial\Delta_2)\subset f(N)\subset\tilde K$,  where $$ sys_1(\tilde K,\tilde g):=\inf\{Length_{\tilde g}(\gamma):\gamma:\mathbb S^1\to \tilde K,\,[\gamma]\neq 0\in \pi_1(\tilde X)\}.$$
\end{remark}

Now we are ready to prove Theorem \ref{Thm: vanishing}.
\begin{proof}[Proof of Theorem \ref{Thm: vanishing}]
Since we work with rational coefficient, we have 
$$f_*([M])=\frac{1}{2}f_*([\check M])$$ when $M$ is non-orientable and $\check M$ is a two-sheeted orientable covering of $M$. This suggests that we only need to deal with the case when $M$ is orientable.

We argue by contradiction. Suppose that $f_*([M])\neq 0$ in $H_n(X,\mathbb Q)$, then our goal is to deduce a contradiction from the fact that $M$ admits a smooth metric $g$ with positive scalar curvature. Up to scaling we can always assume $R(g)\geq 2$ from the compactness of $M$.

First we point out that we can do some simplifications. Up to a homotopy we may assume $f$ to be smooth without loss of generality. Moreover, we may assume $X$ to be orientable. Otherwise we take $p_X:(\hat X,\hat x)\to (X,x)$ to be the orientable covering of $X$ and $p_M:(\hat M,\hat m)\to (M,m)$ to be the covering of $M$ with $m\in f^{-1}(x)$ such that 
    $$f_*\left((p_M)_*(\pi_1(\hat M,\hat m))\right)=(p_X)_*(\pi_1(\hat X,\hat x)).$$
    Now we can lift $f$ to the map $\hat f:(\hat M,\hat m)\to (\hat X,\hat x)$. Clearly we have the following commutative diagram
    \begin{equation*}
\xymatrix{(\hat M,\hat m)\ar[r]^{\hat f}\ar[d]^{p_M}&(\hat X,\hat x)\ar[d]^{p_X}\\
(M,m)\ar[r]^{f}&(X,x).}
\end{equation*}
Notice that $p_M:(\hat M,\hat m)\to (M,m)$ is a finite covering and so $f_*((p_M)_*([\hat M]))$ is non-zero in $H_n(X,\mathbb Q)$. From the commutative diagram above we see that $\hat f_*([\hat M])$ is non-zero in $H_n(\hat X,\mathbb Q)$. In this case we can use the map $\hat f:\hat M\to \hat X$ to deduce a contradiction instead.

Next we divide the argument into two cases.

{\it Case 1. $X$ is closed.} Recall that we have $f_*([M])\neq 0$ in $H_n(X,\mathbb Q)$. It follows from the universal coefficient theorem that $$H_n(X,\mathbb Q)=H_n(X,\mathbb Z)\otimes \mathbb Q.$$ In particular, $f_*([M])$ is non-zero in $H_n(X,\mathbb Z)$. From the Poincar\'e duality we can find a cohomology $1$-class $\alpha$ in $H^1(X,\mathbb Z)$ such that $[X]\frown \alpha=f_*([M])$. It follows from the universal coefficient theorem that the class $\alpha$ can be viewed as a non-zero homomorphism $\alpha:\pi_1(X,x)\to \mathbb Z$. As a consequence, we can find an embedded closed curve $\gamma:(\mathbb S^1,1)\to (X,x)$ such that $\alpha([\gamma])\neq 0$. In particular, $\gamma$ has non-zero intersection number with $f_*([M])$. Without loss of generality we can always assume that $f$ and $\gamma$ intersect transversally.

From our construction it is easy to see that the map 
$$\gamma_*:\pi_1(\mathbb S^1,1)\to \pi_1(X,x)$$ is injective, so its image is isomorphic to $\mathbb Z$. Consider the covering $$\tilde p_X:(\tilde X,\tilde x)\to (X,x)$$ such that
\begin{equation}\label{Eq: tilde X}
(\tilde p_X)_*(\pi_1(\tilde X,\tilde x))=k\gamma_*(\pi_1(\mathbb S^1,1))
\end{equation}
for some positive integer $k$ to be determined later.
Then $\tilde X$ is an aspherical manifold with fundamental group $\mathbb Z$ and so it is homotopy equivalent to the circle $\mathbb S^1$. In particular, we have $H_n(\tilde X,\mathbb Z)=0$. 

In the following, we fix a smooth metric $g_X$ on $X$ such that the map $ f:( M, g)\to ( X, g_X)$ is distance-decreasing. Denote $\tilde g_X$ to be lifted metric of $g_X$ on $\tilde X$. We claim that $sys_1(\tilde X,\tilde g_X)>6C(n-1)$ holds for some integer $k$. 

Before we give a proof for this claim, we make a basic discussion on deck transformations. To distinguish we shall denote the pointed cover satisfying \eqref{Eq: tilde X} by $(\tilde X_k,\tilde x_k)$. Consider the universal covering $$\bar p_X:(\bar X,\bar x)\to (X,x).$$ 
From lifting this induces the universal coverings $\bar p_k:(\bar X,\bar x)\to (\tilde X_k,\tilde x_k)$ for all $k$. We denote $\mathcal D_k$ to be the deck transformation group of the universal covering $\bar X\to \tilde X_k$. It follows from \cite[Proposition 1.39 and its proof]{Hat02} that we have the group isomorphism
$$
\varphi_k:\pi_1(\tilde X_k,\tilde x_k)\to \mathcal D_k,\,[\tilde \gamma]\mapsto \tau_{[\tilde \gamma]},$$
where $\tau_{[\tilde \gamma]}$ is the unique deck transformation mapping $\bar x$ to $\bar x_{[\tilde \gamma]}=\bar\gamma(1)$, where $\bar\gamma:[0,1]\to \bar X$ is the lift of $\tilde \gamma$ with $\bar\gamma(0)=\bar x$. In particular, the deck transformation group $\mathcal D$ is isomorphic to $\mathbb Z$ and in particular we can take a generator $\mathcal T$ of $\mathcal D$. Notice from \eqref{Eq: tilde X} that $(\tilde X_k,\tilde x_k)$ is actually a cover of $(\tilde X_1,\tilde x_1)$ which sends $\pi_1(\tilde X_k,\tilde x_k)$ to $k\pi_1(\tilde X_1,\tilde x_1)$. Then it follows from the correspondence between $\pi_1(\tilde X_k,\tilde x_k)$ and $\mathcal D_k$ given by isomorphism $\varphi_k$ that $\mathcal D_k$ is the deck transformation subgroup $k\mathcal D$.

Suppose by contradiction that $sys_1(\tilde X,\tilde g_X)\leq 6C(n-1)$ for all $k$, where $\tilde X$ is the cover of $X$ satisfying the relation \eqref{Eq: tilde X}. By definition of the systole, we can find a homotopically non-trivial closed curve $\tilde \gamma: \mathbb S^1\to \tilde X$ with length no greater than $6C(n-1)+1$. By lifting $\tilde \gamma$ we can obtain a curve $\bar\gamma:[0,1]\to\bar X$ with the same length such that $\bar \gamma(1)$ lies in $ \mathcal D_k\bar\gamma(0)$. It follows from previous discussion that we can write $\bar\gamma(1)=(\tau k\mathcal T)(\bar\gamma(0))$ for some $\tau\in\mathbb Z\setminus\{0\}$. To sum up for each $k$ we can find a point $\bar x_k$ and an integer $s_k$ with $|s_k|\geq k$ such that $\dist(\bar x_k,(s_k\mathcal T)(\bar x_k))\leq 6C(n-1)+1$. Since $X$ is compact, up to deck transformations we can assume that all $\bar x_k$ lie in a fixed compact subset $\bar K_o$ of $\bar X$ with $\bar p_X(\bar K_o)=X$. Fix a point $\bar x_o\in \bar K_o$. Then we have
\[
\begin{split}\dist(\bar x_o,(s_k\mathcal T)(\bar x_o))\leq &\dist(\bar x_o,\bar x_k)+\dist((s_k\mathcal T)(\bar x_o),(s_k\mathcal T)(\bar x_k))\\
&\qquad+\dist(\bar x_k,(s_k\mathcal T)(\bar x_k))\leq D=D(\bar K_o,n).
\end{split}\]
Since we have $s_k\to\infty$ as $k\to \infty$, there are infinitely many points in $\mathcal D\bar x_o$ contained in a fixed compact subset of $\bar X$, which is impossible since the group action of $\mathcal D$ on $\bar X$ has to be discrete.

Define
$$\tilde M=\{(m',\tilde x')\in M\times\tilde X\,|\,f(m')=\tilde p_X(\tilde x')\}$$
with the induced topology. Denote $G=(f,\tilde p_X):M\times \tilde X\to X\times X$. At each point $(m',\tilde x')$ we have
$$
0\oplus T_{\tilde p_X(\tilde x')}X\subset \mathrm dG\left(T_{(m',\tilde x')}(M\times \tilde X)\right)
$$
since $\tilde p_X$ is a covering map. This implies that the map $G$ is transversal to the diagonal subset $\Delta_X=\{(x,x):x\in X\}\subset X\times X$ and so $\tilde M$ is a properly embedded submanifold of $M\times \tilde X$. In particular,  the inclusion $\tilde M\hookrightarrow M\times \tilde X$ is proper. Denote $\tilde p_M:\tilde M\to M$ and $\tilde f:\tilde M\to \tilde X$ to be the restriction of the projection maps, from $M\times \tilde X$ to its components, on $\tilde M$. Since $M$ is closed, the projection map $M\times \tilde X\to\tilde X$ is proper. Then the map $\tilde f$ is proper since $\tilde f$ is the composition of $\tilde M\hookrightarrow M\times\tilde X\to \tilde X$. We also point out that $\tilde p_M:\tilde M\to M$ is actually a covering map. To see this we take a fundamental domain $U$ of $X$ such that $\tilde p_X^{-1}(U)=\sqcup_i \tilde U_i$ where $\tilde p_X:\tilde U_i\to U$ are homeomorphisms for all $i$. Then for any domain $V$ of $ M$ with $f(V)\subset U$ it is straightforward to verify
$$\tilde p_M^{-1}(V)=\bigsqcup_i\tilde V_i\mbox{ with }\tilde V_i=\{(v,(\tilde p_X|_{\tilde U_i})^{-1}\circ f(v)):v\in V\}$$
and that $\tilde p_M:\tilde V_i\to V$ is a homeomorphism.
By definition $\tilde p_M$ is a covering map. In particular, we can lift $g$ to a metric $\tilde g$ on $\tilde M$ satisfying $R(\tilde g)\geq 2$.

By lifting we can obtain a map $\tilde \gamma:(\mathbb S^1,1)\to (\tilde X,\tilde x)$ such that $$\tilde\gamma_*(\pi_1(\mathbb S^1,1))=\pi_1(\tilde X,\tilde x).$$ To sum up we have the following commutative diagram
  \begin{equation*}
\xymatrix{(\tilde M,\tilde f^{-1}(\tilde x))\ar[r]^{\tilde f}\ar[d]^{\tilde p_M}&(\tilde X,\tilde x)\ar[d]^{\tilde p_X}&(\mathbb S^1,1)\ar[l]_{\tilde \gamma}\ar[d]^{z^k}\\
(M,m)\ar[r]^{f}&(X,x)&(\mathbb S^1,1)\ar[l]_{\gamma}.}
\end{equation*}
Let us verify that the (geometric) intersection number of $\tilde f$ and $\tilde \gamma$ equals to $kf_*([M])\cdot [\gamma]$. To see this we start with the fact that if we have $f(m_0)=x_0$ then $\tilde f$ maps $\tilde p_M^{-1}(m_0)$ one-to-one on $\tilde p_X^{-1}(x_0)$. Notice that the curve $\tilde \gamma$ passes through exactly $k$ points in $\tilde p_X^{-1}(x_0)$. Notice that the intersections of $\tilde f$ and $\tilde \gamma$ at these points all have the same sign and so this yields the desired conclusion. Clearly we can pick up just one component of $\tilde M$ such that its image under $\tilde f$ has non-zero intersection with $\tilde\gamma$. For convenience, we still denote it by $\tilde M$.

Now let us deduce the desired contradiction. From the fact $H_n(\tilde X,\mathbb Z)=0$ and also the non-zero intersection above we conclude the non-compactness of $\tilde M$. {\color{blue} } Pick up a smoothing $\tilde \rho$ of the distance function to the point $\tilde m$ in $(\tilde M,\tilde g)$ with $\Lip\tilde \rho<1$, where $\tilde g$ is the  lifted metric of $g$ on $\tilde M$. It follows from Gromov's $\mu$-bubble method (see \cite{Gro20, GZ21, Zhu23} for instance) that for $L$ large enough we can find a smooth hypersurface $\Sigma\subset \tilde \rho^{-1}([T,2T])$ separating $\tilde \rho^{-1}(T)$ and $\tilde \rho^{-1}(2T)$ with $\mathbb T^1$-stabilized scalar curvature $R\geq 1$.  Let $\tilde C$ be the absolute constant coming from Lemma \ref{lem2}. Denote $\Omega$ to be the bounded region in $\tilde M$ enclosed by $\Sigma$. From the properness of $\tilde f$ we can take $L$ large enough such that $\tilde f(\tilde M-\Omega)$ is away from the $\tilde C$-neighborhood of $\tilde \gamma(\mathbb S^1)$. In particular, $\Omega$ has non-zero intersection number with $\tilde \gamma$. Recall that $\tilde f:(\tilde M,\tilde p_M^*g)\to (\tilde X,\tilde p_X^*g_X)$ is distance-decreasing (as a lift of the map $f$) and so is the restricted map $\tilde f|_\Sigma$. Using Lemma \ref{lem2} we can find an $n$-chain $\sigma$ in $\Tilde{X}$ with $\partial\sigma=\tilde f_\#\Sigma$ such that the support of $\sigma$ avoids $\tilde\gamma(\mathbb S^1)$. Then $\Omega-\sigma$ is an $n$-cycle with non-zero intersection with $\tilde \gamma$, which contradicts to $H_{n}(\tilde X,\mathbb Z)=0$.

{\it Case 2. $X$ is non-compact.} If there exists a closed curve $\gamma$ with non-zero intersection number with $f_*([M])$, then one could follow the same argument as in Case 1 to derive the result (the only modification would be that we fix a compact subset $K$ of $X$ and establish the estimate $sys_1(\tilde K,\tilde g_X)>6C(n-1)$ for $\tilde K=\tilde p_X^{-1}(K)$ due to Remark \ref{Rmk: compact}). Therefore, we can make the assumption that each closed curve has intersection number zero with $f_*([M])$ in the following discussion.

Fix a complete metric $g_X$ on $X$ such that the map $f:(M,g)\to (X,g_X)$ is distance-decreasing. Take a compact and increasing exhaustion $\{U_i\}_{i=1}^\infty$ with $f(M)\subset U_1$. It is clear that $f_*([M])\ne 0$ in $H_n(U_i,\mathbb Z)$. By Poincar\'e duality and the universal coefficient theorem we can choose a curve $\gamma_i$ with $\partial\gamma_i\subset \partial U_i$ which has non-zero intersection number with $f_*([M])$. By our assumption $\gamma_i$ cannot be closed and so we can denote the end-points by $a_i$ and $b_i$. For the same reason the end-points $a_i$ and $b_i$ must lie in different components of $\partial U_i$ (otherwise we are able to produce a closed curve with non-zero intersection with $f_*([M])$). Let $A_i$ and $B_i$ be the components of $\partial U_i$ with $a_i\in A_i$ and $b_i\in B_i$. Choose a minimizing geodesic $l_i$ in $U_i$ connecting $A_i$ and $B_i$. Notice that $l_i-\gamma_i$ is homologuous to a closed curve in $(U_i,\partial U_i)$. Using our assumption and the diffential-topology interpretation of the intersection number we see 
\begin{align*}
  [l_i]\cdot f_*([M])=[\gamma_i]\cdot f_*([M])\ne 0.
\end{align*}
This implies that each $l_i$ always intersects with the fixed compact set $f(M)$. Thus, by taking the limit of $l_i$ as $i\to\infty$ up to 
a subsequence we can construct a geodesic line $l$ with non-zero intersection number with $f_*([M])$.

Now let $x=f(m)$ be a point in $l$ and we consider the universal cover $(\bar{X},\bar x)$ of $(X,x)$. As before, we define
$$\bar M=\{(m',\bar x')\in M\times\bar X\,|\,f(m')=\bar p_X(\bar x')\}$$
and restricted projection maps
$$\bar p_M:\bar M\to M\mbox{ and }\bar f:\bar M\to \bar X.$$
As before, we can verify that the map $\bar f$ is proper and that $\bar p_M:\bar M\to M$ is a covering map in the same way. In particular, we can lift $g$ to a metric $\bar g$ on $\bar M$ satisfying $R(\bar g)\geq 2$. By lifting we can construct a geodesic line $\bar{l}:\mathbb R\to \bar X$ and we obtain the following commutative diagram:
\begin{equation*}
\xymatrix{(\bar M,\bar f^{-1}(\bar x))\ar[r]^{\bar f}\ar[d]^{\bar p_M}&(\bar X,\bar x)\ar[d]^{\bar p_X}&(\mathbb R,0)\ar[l]_{\bar l}\ar[d]^{id}\\
(M,m)\ar[r]^{f}&(X,x)&(\mathbb R,0)\ar[l]_{l}.}
\end{equation*}
As before, one easily checks that the (geometric) intersection number of $\bar{l}$ and the map $\bar{f}: \bar{M}\longrightarrow \bar{X}$ equals that of $l$ and $f$. Again we can pass to a component of $\bar M$, still denoted by $\bar M$,
which has non-zero intersection number with $\bar l$. In particular, we conclude that $\bar{M}$ has to be non-compact, otherwise $\bar f_*([\bar M])$ is null-homologous from the contractibility of $\bar{X}$, violating the non-trivially intersecting condition.

Let $U_{\rho}$($\Tilde{l}$) be the $\rho$-neighbourhood of $\bar{l}$ in $(\bar{X},\bar p_X^*g_X)$. We claim that the preimage $\bar{f}^{-1}(U_{\rho}(\bar{l}))$ is compact for all $\rho>0$. In fact, since $l:\mathbb R\to (X,g_X)$ is a geodesic line, there exists a constant $T>0$ such that
\begin{align*}
    \dist_X(l(t),f(M))>\rho\mbox{ for all } |t|>T.
\end{align*}
Therefore we obtain
\begin{align*}
    \dist_{\bar X}(\bar{l}(t),\bar{f}(\bar{M}))>\rho\mbox{ for all } |t|>T
\end{align*}
from the fact that the covering map is a local isometry. In particular, we have
\begin{align*}
    \bar{f}^{-1}(U_{\rho}(\bar{l})) = \bar{f}^{-1}\left(U_{\rho}(\bar l|_{[-T,T]})\right),
\end{align*}
which is compact due to the properness of $\bar{f}$. Then we could follow the same argument as in Case 1 to obtain a contradiction.
\end{proof}

\begin{proof}[Proof for Corollary \ref{Cor: 5D}]
    When $k=2$ or $3$ it follows directly from Theorem \ref{Thm: 2D}, and the case when $k=5$ was handled by \cite{Gro20, CL20, CLL23}. The remaining case when $k=4$ follows from our reduction Theorem \ref{Thm: vanishing} and the verification of the generalized filling radius conjecture from the works \cite{CL20,LM23}.
\end{proof}

Now we prove the aspherical splitting theorem.

\begin{proof}[Proof of Theorem \ref{Thm: splitting}]
    If $(M,g)$ is Ricci-flat, then the proof is completed by Cheeger-Gromoll splitting theorem \cite{CG71}. Otherwise, it follows from Kazdan's deformation \cite{Kaz82} that we are able to find a new complete metric on $M$ with positive scalar curvature, still denoted by $g$ for simplicity. Our goal is to show that this case cannot happen. Doing product with circle or torus we may assume $n=4$.

Since $M$ has two ends, we can find a proper and surjective smooth function $\rho:M\to (-\infty,+\infty)$ with $\Lip \rho <1$. Denote 
$$K=\rho^{-1}([-1,1])\mbox{ and } R_0=\inf_K R(g).$$
Then it is not difficult to construct a smooth function 
$$
h:(-T_0,T_0)\to (-\infty,+\infty)
$$
such that 
\begin{itemize}
\item $h(t)\to \pm\infty$ as $t\to \mp T_0$ and $h'(t)<0$ for all $t\in (-T_0,T_0)$;
\item we have 
$$R_0\chi_{[-1,1]}+2h'+\frac{n+1}{n}h^2>0$$.
\end{itemize}

Let us set the $\mu$-bubble problem as follows. Denote $\Omega_0=\{\rho <0\}$ and consider the functional
$$
\mathcal A^h(\Omega)=\mathcal H^{n}_g(\partial\Omega)-\int_{M}(\chi_\Omega-\chi_{\Omega_0})h\circ \rho\,\mathrm d\mathcal H^{n+1}_g
$$
defined for regions $\Omega$ belonging to the class
\begin{equation*}
\mathcal C=\left\{\begin{array}{c}\mbox{Caccioppoli sets }\Omega\subset M\mbox{ such that} \\
\mbox{$\Omega\Delta \Omega_0$ has compact closure in $\{-T_0<\rho<T_0\}$}
\end{array}\right\}.
\end{equation*}
It follows from \cite{Zhu21} that we can find a smooth minimizer $\Omega_{min}$ of $\mathcal A^h$ in $\mathcal C$ (no regularity issue occurs due to the fact $n+1= 5$). Consider the variation of $\Omega_{min}$ along a compactly supported vector field $X$ with $X=\phi\nu$ on $\partial\Omega_{min}$, where $\nu$ is the outward unit normal of $\partial\Omega_{min}$. We can compute the first variation formula
$$
\delta\mathcal A^h(\phi)=\int_{\partial\Omega_{min}}(H-h\circ \rho)\phi\,\mathrm d\mathcal H^{n}_g=0
$$
where $H$ is the mean curvature of $\partial\Omega_{min} $ with respect to $\nu$. From arbitrary choice of $\phi$ we obtain $H=h\circ \rho$ on $\partial\Omega_{min}$. From the second variation formula we see
\begin{equation}\label{Eq: second variation}
\begin{split}
&2\int_{\partial\Omega_{min}}|\nabla \phi|^2\,\mathrm d\mathcal H^{n-1}_g\geq\\
&\int_{\partial\Omega_{min}}\left(R(g)-R(i^*g)+H^2+|A|^2+2\partial_\nu(h\circ \rho)\right)\phi^2\,\mathrm d\mathcal H^{n}_g,
\end{split}
\end{equation}
where $R(i^*g)$ and $A$ are the scalar curvature and second fundamental form of $\partial\Omega_{min}$ respectively. Using the facts $H=h\circ \rho$, $\Lip \rho<1$ and $h'(t)<0$ we deduce
$$
R(g)+H^2+|A|^2+2\partial_\nu(h\circ \rho)\geq R(g)+2h'\circ \rho+\frac{n+1}{n}(h\circ \rho)^2>0,
$$
where the right hand side is no less than $\left(R_0\chi_{[-1,1]}+2h'+\frac{n+1}{n}h^2\right)\circ\rho$ and so it is positive.
This implies that the operator 
$$-\Delta+\frac{1}{2}R(i^*g)$$
is positive for each component of $\partial\Omega_{min}$ and the same thing holds for the conformal Laplacian operator. In particular, we can find a smooth metric on $\partial\Omega_{min}$ with positive scalar curvature. 

Notice that the boundary $\partial\Omega_0$ represents a non-zero $\mathbb Q$-homology class in the orientable manifold $M$ since we can construct a line intersecting $\partial\Omega_0$ only once (after counting algebraic multiplicity). From the definition of the class $\mathcal C$ we see that $\partial\Omega_{min}$ is homologous to $\partial\Omega_0$ and so represents a non-zero $\mathbb Q$-homology class as well. This contradicts to Corollary \ref{Cor: 5D} since it follows from previous discussion that $\partial\Omega_{min}$ admits a smooth metric with positive scalar curvature.
\end{proof}

\appendix
\section{vanishing theorem in dimension three and four}\label{Append}
\begin{theorem}\label{Thm: 2D}
     Let $M^n$, $n\in\{2,3\}$, be a closed manifold admitting a smooth metric with positive scalar curvature and $X$ be an aspherical topological space. Then for any continuous map $f:M\to X$ we have $f_*([M])=0$ in $H_n(X,\mathbb Q)$.
\end{theorem}

The proof of Theorem \ref{Thm: 2D} has already appeared in \cite{Wang19, HHLS23} but we still present a proof for convenience of the reader.

\begin{proof}
    Up to lifting we may assume $M$ to be orientable. In dimension two, the only orientable closed surface is the sphere. Therefore, $f_*([M])$ lies in the image of the Hurewicz map, which has to be zero from the aspherical property of $X$. In dimension three, $M$ must be diffeomorphic to the connected sum
    $$(\mathbb S^3/\Gamma_1)\#\cdots \#(\mathbb S^3/\Gamma_k)\#l(\mathbb S^2\times \mathbb S^1).$$
In particular, $M$ has a finite cover $\tilde M$ diffeomorphic to the connected sum of finitely many $(\mathbb S^2\times \mathbb S^1)$'s, which comes from $0$-surgeries to $\mathbb S^3$. From the aspherical property of $X$ we can perform the reverse $2$-surgeries to $\tilde M$ in $X$ to see $(f\circ\pi)([\tilde M])=0$ in $H_3(X,\mathbb Q)$, where $\pi:\tilde M\to M$ is the covering map. Clearly $f_*([M])=c(f\circ\pi)([\tilde M])=0$ for some rational constant $c$.
\end{proof}

\end{document}